\documentclass[a4paper,11pt,reqno]{amsart}
\usepackage{amsfonts,amsmath,amssymb,amsthm}
\usepackage{graphicx}
\usepackage{fixmath}
\usepackage{hyperref}
\usepackage[english]{babel}

\usepackage{setspace}
\usepackage{subfig}
\usepackage{url}
\usepackage{booktabs}
\usepackage{ifthen}
\usepackage{wrapfig}

\usepackage{enumerate}

\usepackage{tikz}
\usetikzlibrary{calc}
\usetikzlibrary{decorations.pathreplacing}

\usepackage[T2A,T1]{fontenc}

\newtheorem{thm}{Theorem}[section]
\newtheorem{lem}[thm]{Lemma}
\newtheorem{prop}[thm]{Proposition}
\newtheorem{cor}[thm]{Corollary}
\newtheorem{obs}[thm]{Observation}

\newtheorem{defn}[thm]{Definition}

\newtheorem{conj}[thm]{Conjecture}

\newtheorem{claim}[thm]{Claim}
\newtheorem{q}{Question}

\numberwithin{equation}{section}

\begin{document}
 
\thispagestyle{plain}

\title{Improved bounds for relaxed graceful trees}

\author{Christian Barrientos}
\address{Christian Barrientos  \texttt{chr\_barrientos@yahoo.com}}

\author{Elliot Krop}
\address{Elliot Krop  \texttt{ElliotKrop@clayton.edu}}

\address{Department of Mathematics \\
Clayton State University \\
Morrow, GA 30260, USA}

\begin{abstract}
We introduce left and right-layered trees as trees with a specific representation and define the excess of a tree. Applying these ideas, we show a range-relaxed graceful labeling which improves on the upper bound for maximum vertex label given by Van Bussel. For the case when the tree is a lobster of size $m$ and diameter $d$, the labeling produces vertex labels no greater than $\frac{3}{2}m-\frac{1}{2}d$. Furthermore, we show that any lobster $T$ with $m$ edges and diameter $d$ has an edge-relaxed graceful bipartite labeling with at least $\max\{\frac{3m-d+6}{4},\frac{5m+d+15}{8}\}$ of the edge weights distinct, which is an improvement on a bound given by Rosa and \v{S}ir\'{a}\v{n} on the $\alpha$-size of trees, for $d<\frac{m+22}{7}$ and $d>\frac{5m-65}{7}$. We also show that there exists an edge-relaxed graceful labeling (not necessarily bipartite) with at least $\max\left\{\frac{3}{4}m+\frac{d-\nu}{8}+\frac{3}{2},\nu\right\}$ of the edge weights distinct, where $\nu$ is twice the size of a partial matching of $T$. This is an improvement on the best known gracesize bound, for certain values of $\nu$ and $d$. We view these results as a step towards Bermond's conjecture that all lobsters are graceful.
\\[\baselineskip] 
\\[\baselineskip]
\end{abstract}

\date{\today}

\maketitle

\section{Introduction} 
The graceful tree conjecture (GTC), first formulated in 1966 by Rosa \cite{Rosa}, has played a role as the point of origin for most of the questions and results in graph labeling. At this time, Gallian's survey shows two-thousand-one-hundred-and-twelve references \cite{Gallian} most of which can claim the GTC as its root. 

The question, as with many popular open combinatorial problems, is easy to state. 

Let $T$ be a tree on $n$ vertices. Define a \emph{weight} on an edge as the absolute difference of the labels of its incident vertices. 
\begin{conj}\cite{Rosa}
It is possible to label the vertices of $T$ uniquely from $0$ to $n-1$, so that the set of weights of $T$ is $\{1,2, \dots, n-1\}$.
\end{conj}

Every tree that accepts the conjectured labeling is known as a \emph{graceful tree}. Though many labeling schemes exist to prove specific families of trees graceful, there has been little progress in proving the conjecture even for a robust class of ``shallow'' trees. More precisely, using a definition of \emph{tree distance} that first appeared in \cite{Morgan}, let $P$ be a longest path in $T$ and call $T$ a $k$-distant tree if all of its vertices are a distance at most $k$ from $P$. Paths ($0$-distant trees) and caterpillars ($1$-distant trees) were shown to be graceful in \cite{Rosa}. However, the conjecture is unknown for any trees with higher tree distance. In fact, the problem for $2$-distant trees, or \emph{lobsters}, is a well-known conjecture of Bermond \cite{Bermond}.

\begin{conj}\cite{Bermond}
Every lobster is graceful.
\end{conj}

Approximate approaches to the GTC were introduced by Golomb \cite{Golomb} in 1972. A particularly natural relaxation of the graceful condition, defined by Van Bussel \cite{VB} is the following: Let $G$ be a graph with vertex set $V$ and edge set $E$, $f(V)\rightarrow \mathbb{N}$ an injective map to $\{0, \dots, m'\}$ for some $m' > m= \left|E\right|$ (producing the vertex labels), and $g:E\rightarrow \mathbb{N}$ an injective map to $\{1, \dots, m''\}$ for some $m''\geq m= \left|E\right|$ defined by $g(uv)=\left|f(u)-f(v)\right|$ (producing the weights of $uv$ under $f$). The map $f$ is called a \emph{range-relaxed labeling}.

The ``best'' bound on the maximum vertex labels in a range-relaxed labeling of trees was given by Van Bussel \cite{VB}.

\begin{thm}\cite{VB}
Every tree $T$ on $m$ edges has a range-relaxed graceful labeling $f$ with vertex labels in the range, $0, \dots, 2m-\mbox{diam}(T)$.
\end{thm}

Van Bussel also asked the following:

\begin{q}\cite{VB}
For a tree $T$ on $m$ edges and $n = m + 1$ vertices, is there any $\varepsilon > 0$ for
which we can guarantee a range-relaxed graceful labeling within the range $(2 - \varepsilon)m$?
\end{q}

Another relaxation of the graceful condition was defined by Rosa and \v{S}ir\'{a}\v{n} \cite{RS}. We first define terms in the language of \cite{VB}. Let $G$ be a graph with vertex set $V$ and edge set $E$, $f(V)\rightarrow \mathbb{N}$ an injective map to $\{0, \dots, m\}$ where $m= \left|E\right|$ (producing the vertex labels), and $g(E)\rightarrow \mathbb{N}$ a map to $\{1, \dots, m\}$ defined by $g(uv)=\left|f(u)-f(v)\right|$ (producing the induced edge weights under $f$). Such a map $f$ is called an \emph{edge-relaxed graceful labeling}. 

\begin{thm}\cite{RS}
Every tree of size $m$ has an edge-relaxed graceful labeling with at least $\frac{5}{7}(m+1)$ distinct weights.
\end{thm}

The \emph{gracesize} of a tree $T$, $gs(T)$, is the largest number of distinct edge weights in any edge-relaxed labeling of $T$. With this notation, the above theorem presents the lower bound $gs(T)\geq \frac{5}{7}(m+1)$.

A labeling $f$ is said to be \emph{bipartite}, if there exists an integer $c$ such that for any edge $uv$, either $f(u)\leq c < f(v)$ or $f(v)\leq c < f(u)$. A bipartite labeling that is graceful is called an $\alpha$\emph{-labeling}. As with the concept of gracesize, the $\alpha$\emph{-size} of a tree $T$, $\alpha(T)$, is the largest number of distinct edge weights in any bipartite labeling of $T$, i.e., $\alpha(T)=\max\{\varepsilon(f)$: $f$ is a bipartite labeling of $T\}$, where $\varepsilon(f)$ is the cardinality of the set of weights induced by $f$ on the edges of $T$.

Rosa and \v{S}ir\'{a}\v{n} \cite{RS} proved that for any tree $T$ with $m$ edges, $\frac{5}{7}(m+1)\leq \alpha(T) \leq \frac{5m+9}{6}$.

It should be noted that the lower bound for gracesize was taken from that for $\alpha$-size, as it was not known how to use non-bipartite labelings to improve this bound.

Several results have come out in the intervening years on the lower bounds for the $\alpha$-size of certain restricted families of trees. For the case of trees with maximum degree $3$, Bonnington and \v{S}ir\'{a}\v{n} \cite{BS} showed that $\alpha(T)\geq \frac{5}{6}(m+1)$. Brankovic, Rosa, and \v{S}ir\'{a}\v{n} \cite{BRS} later improved this bound to $\alpha(T)\geq \left\lfloor\frac{6}{7}(m+1)\right\rfloor-1$. 

A direction with some positive advancement has been for trees that have a perfect matching. In 1999, Broersma and Hoede \cite{BH} showed a surprising equivalence between the graceful labelings and a more restrictive labeling on trees containing a perfect matching.

Let $T$ be a tree of order $n$ with a perfect matching $M$. A graceful labeling of $T$, which additionally satisfies the property that for any edge in $M$, the pair of vertices incident to that edge must have a label sum of $n-1$, is called a \emph{strongly graceful} labeling. For any tree $T$ with a perfect matching $M$, the tree resulting from the contraction of the edges of $M$ is called the \emph{contree} of $T$.

\begin{thm}\cite{BH}\label{BHEquivalence}
Every tree is graceful if and only if every tree containing a perfect matching is strongly graceful.
\end{thm}

Furthermore, the authors proved

\begin{cor}\cite{BH}\label{BHLobsters}
Every tree containing a perfect matching and having a caterpillar for its contree is strongly graceful.
\end{cor}

Although the following theorem is easily implied by the proof of Theorem \ref{BHEquivalence}, and is an immediate consequence of the above corollary, the authors did not state it. Superdock \cite{Superdock} was the first to mention that connection. Morgan proved the following by an explicit construction.

\begin{thm}\cite{Morgan}\label{Morgan}
All lobsters with perfect matchings are graceful.
\end{thm}

Our result, related to lobsters, requires the flexibility of the labeling in the proof of Theorem \ref{BHEquivalence} and Corollary \ref{BHLobsters}. We will call the graceful labelings of lobsters with a perfect matching described in \cite{BH}, the \emph{Broersma-Hoede labeling}, or \emph{BH labeling}. For completeness, we review this labeling in section 2.2.

In this paper, we introduce the concept of the \emph{excess} of a tree and use it to improve the bounds on the ranges for both range-relaxed and edge-relaxed graceful labelings. This allows us to answer Van Bussel's question for lobsters with $\varepsilon = \frac{1}{2}$. In our last result which pertains to lobsters, we use BH labelings to slightly improve the gracesize bound obtained from analyzing the excess.

For basic graph theoretic concepts, we refer the reader to the book \cite{West}.


\section{The Excess of Layered Trees}
Let $T$ be a rooted tree with vertices ordered vertically by distance from the root and the root above all other vertices. For any vertex $v \in V(T)$, let $\gamma(v)$ denote the number of levels in $T$ below $v$, where $v$ has at least one descendant.

Our labeling applies to rooted trees where the root $r$ has been chosen in such a way that $\deg(r)=1$ and $\gamma(r)=d=diam(T)$, i.e. $r$ is a pendant vertex of a maximal path in $T$. We order the vertices within each level according to their degrees and the associated parameter $\gamma$ so that edges do not cross. The level of the root vertex $r$ is denoted $L_0$, and vertices of distance $j>0$ from $r$ are on level $L_j$, represented $j$ levels below $r$. We denote by $u\prec v$ the placement of $u$ to the left of $v$. With this notation, we define the order on each level.

\begin{enumerate}
\item If $u$ and $v$ are siblings of degree one, the order of $u$ and $v$ is arbitrary.
\item If $u$ and $v$ are siblings and $\gamma(u)<\gamma(v)$, then $u \prec v$.
\item If $u$ and $v$ are siblings and $\gamma(u)=\gamma(v)$, and $\deg(u)\geq \deg(v)$, then $u\prec v$.
\item If $u$ and $v$ are siblings and $u\prec v$, and $a$ and $b$ descendants of $u$ and $v$, respectively, on the same level, then $a\prec b$. 
\end{enumerate}

A rooted tree so represented is called a \emph{left-layered tree}.
By this ordering, a path of maximum length is drawn on the right extreme of the picture. In Figure \ref{notbip} we show a left-layered tree on $14$ vertices.

Let $L_j=\{v_i^j:1\leq i\leq n_j\}$ where $n_j=\left|L_j\right|$. We assume that $v_i^j \prec v_{i+1}^j$ for all $1\leq i \leq n_j-1$. Let $i$ be the smallest index such that $dist(v_i^j,v_{i+1}^j)>2$. We define $ex_j$ to be the cardinality of the set $\{v \in L_{j-1}: u\prec v\}$ where $u$ is the parent of $v_i^j$. That is, $ex_j$ counts the number of vertices on level $L_{j-1}$ located on the right side of $u$.

Another way to think of this quantity, which may be helpful, is that in a left-layered tree, $ex_j$ counts the number of consecutive vertices on level $j$ with distance greater than $2$ along with the number of siblings linearly between their parents (aunts and uncles). Define the \emph{excess} of $T$, denoted by $ex(T)$, as
\[ex(T)=\sum_{j=0}^d {ex_j}\]

Notice that if $\left|L_j\right|=1$, $ex_j=0$. If $\left|L_j\right|>1$ and $dist(v_a^j,v_b^j)=2$, for every pair of vertices in $L_j$, $ex_j=0$.

In addition, observe that this parameter not only depends on the left-layered tree, it also depends on the vertex chosen to be the root. In Figure \ref{root} we show this situation by exhibiting two different representations of a left-layered tree of size $32$. In both cases there are $12$ levels. For $1\leq i \leq 11$, $ex_j$ corresponds to the number of black vertices in level $j-1$. We note that for the representation in $(a)$, $ex_0=ex_1=ex_2=ex_3=ex_4=ex_5=0, \, ex_6=1, \, ex_6=1, \, ex_7=2, \, ex_8=3, \, ex_9=1, \, ex_{10}=2,$ and $ex_{11}=2$. Therefore, the excess of the representation in $(a)$ is $11$ while it is $12$ in $(b)$. 

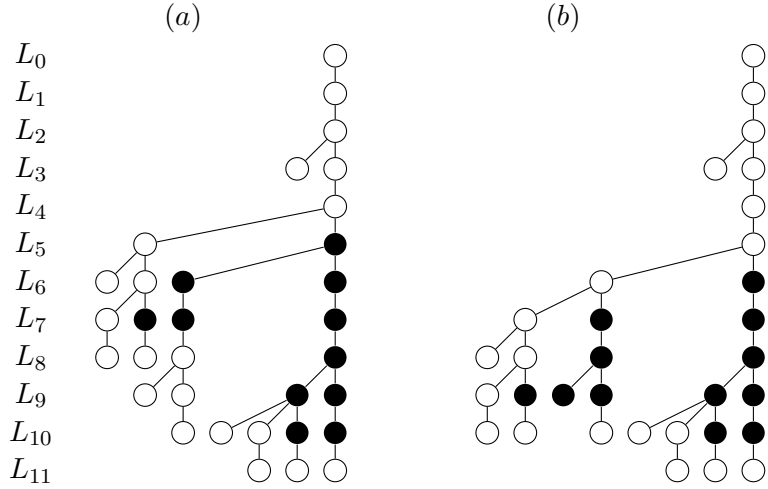
\begin{figure}[ht]
\begin{center}
\begin{tikzpicture}[scale=.5]
\tikzstyle{vert}=[circle,fill=black,inner sep=3pt]
\tikzstyle{overt}=[circle, draw, inner sep=3pt]
  
\node at (-1,12) {$L_0$};
\node at (-1,11) {$L_1$};
\node at (-1,10) {$L_2$};
\node at (-1,9) {$L_3$};
\node at (-1,8) {$L_4$};
\node at (-1,7) {$L_5$};
\node at (-1,6) {$L_6$};
\node at (-1,5) {$L_7$};
\node at (-1,4) {$L_8$};
\node at (-1,3) {$L_9$};
\node at (-1,2) {$L_{10}$};
\node at (-1,1) {$L_{11}$};


\node at (3,13) {$(a)$};

\node[overt, label=left:\tiny{}] (u1) at (7,1) {};
\node[vert, label=left:\tiny{}] (u2) at (7,2) {};
\node[vert, label=left:\tiny{}] (u3) at (7,3) {};
\node[vert, label=left:\tiny{}] (u4) at (7,4) {};
\node[vert, label=left:\tiny{}] (u5) at (7,5) {};
\node[vert, label=left:\tiny{}] (u6) at (7,6) {};
\node[vert, label=left:\tiny{}] (u7) at (7,7) {};
\node[overt, label=left:\tiny{}] (u8) at (7,8) {};
\node[overt, label=left:\tiny{}] (u9) at (7,9) {};
\node[overt, label=left:\tiny{}] (u10) at (7,10) {};
\node[overt, label=left:\tiny{}] (u11) at (7,11) {};
\node[overt, label=left:\tiny{}] (u12) at (7,12) {};

\node[overt, label=left:\tiny{}] (u13) at (6,9) {};

\node[overt, label=left:\tiny{}] (u14) at (2,7) {};

\node[overt, label=left:\tiny{}] (u15) at (1,6) {};
\node[overt, label=left:\tiny{}] (u16) at (2,6) {};
\node[vert, label=left:\tiny{}] (u17) at (3,6) {};

\node[overt, label=left:\tiny{}] (u18) at (1,5) {};
\node[vert, label=left:\tiny{}] (u19) at (2,5) {};
\node[vert, label=left:\tiny{}] (u20) at (3,5) {};

\node[overt, label=left:\tiny{}] (u21) at (1,4) {};
\node[overt, label=left:\tiny{}] (u22) at (2,4) {};
\node[overt, label=left:\tiny{}] (u23) at (3,4) {};

\node[overt, label=left:\tiny{}] (u24) at (2,3) {};
\node[overt, label=left:\tiny{}] (u25) at (3,3) {};
\node[vert, label=left:\tiny{}] (u26) at (6,3) {};

\node[overt, label=left:\tiny{}] (u27) at (3,2) {};
\node[overt, label=left:\tiny{}] (u28) at (4,2) {};
\node[overt, label=left:\tiny{}] (u29) at (5,2) {};
\node[vert, label=left:\tiny{}] (u30) at (6,2) {};

\node[overt, label=left:\tiny{}] (u31) at (5,1) {};
\node[overt, label=left:\tiny{}] (u32) at (6,1) {};

  \draw[color=black] 
(u1)--(u2)--(u3)--(u4)--(u5)--(u6)--(u7)--(u8)--(u9)--(u10)--(u11)--(u12)

(u13)--(u10) (u14)--(u8) (u15)--(u14) (u16)--(u14) (u17)--(u7) (u18)--(u16) (u19)--(u16) (u20)--(u17) (u21)--(u18) (u22)--(u19) (u23)--(u20) (u24)--(u23) (u25)--(u23) (u26)--(u4) (u27)--(u25) (u28)--(u26) (u29)--(u26) (u30)--(u26) (u31)--(u29) (u32)--(u30)

;


\node at (13,13) {$(b)$};

\node[overt, label=left:\tiny{}] (v1) at (18,1) {};
\node[vert, label=left:\tiny{}] (v2) at (18,2) {};
\node[vert, label=left:\tiny{}] (v3) at (18,3) {};
\node[vert, label=left:\tiny{}] (v4) at (18,4) {};
\node[vert, label=left:\tiny{}] (v5) at (18,5) {};
\node[vert, label=left:\tiny{}] (v6) at (18,6) {};
\node[overt, label=left:\tiny{}] (v7) at (18,7) {};
\node[overt, label=left:\tiny{}] (v8) at (18,8) {};
\node[overt, label=left:\tiny{}] (v9) at (18,9) {};
\node[overt, label=left:\tiny{}] (v10) at (18,10) {};
\node[overt, label=left:\tiny{}] (v11) at (18,11) {};
\node[overt, label=left:\tiny{}] (v12) at (18,12) {};

\node[overt, label=left:\tiny{}] (v13) at (17,9) {};

\node[overt, label=left:\tiny{}] (v14) at (14,6) {};

\node[overt, label=left:\tiny{}] (v15) at (12,5) {};
\node[vert, label=left:\tiny{}] (v16) at (14,5) {};

\node[overt, label=left:\tiny{}] (v17) at (11,4) {};
\node[overt, label=left:\tiny{}] (v18) at (12,4) {};
\node[vert, label=left:\tiny{}] (v19) at (14,4) {};

\node[overt, label=left:\tiny{}] (v20) at (11,3) {};
\node[vert, label=left:\tiny{}] (v21) at (12,3) {};
\node[vert, label=left:\tiny{}] (v22) at (13,3) {};
\node[vert, label=left:\tiny{}] (v23) at (14,3) {};
\node[vert, label=left:\tiny{}] (v24) at (17,3) {};

\node[overt, label=left:\tiny{}] (v25) at (11,2) {};
\node[overt, label=left:\tiny{}] (v26) at (12,2) {};
\node[overt, label=left:\tiny{}] (v27) at (14,2) {};
\node[overt, label=left:\tiny{}] (v28) at (15,2) {};
\node[overt, label=left:\tiny{}] (v29) at (16,2) {};
\node[vert, label=left:\tiny{}] (v30) at (17,2) {};

\node[overt, label=left:\tiny{}] (v31) at (16,1) {};
\node[overt, label=left:\tiny{}] (v32) at (17,1) {};

  \draw[color=black] 
 (v1)--(v2)--(v3)--(v4)--(v5)--(v6)--(v7)--(v8)--(v9)--(v10)--(v11)--(v12)

(v13)--(v10) (v14)--(v7) (v15)--(v14) (v16)--(v14) (v17)--(v15) (v18)--(v15) (v19)--(v16) (v20)--(v18) (v21)--(v18) (v22)--(v19) (v23)--(v19) (v24)--(v4) (v25)--(v20) (v26)--(v21) (v27)--(v23) (v28)--(v24) (v29)--(v24) (v30)--(v24) (v31)--(v29) (v32)--(v30)

;

\end{tikzpicture}
\caption{Left-layered trees and their excess}
\label{root}
\end{center}
\end{figure}

We define $ex'_j$ to be the number of consecutive vertices on level $j$ with distance greater than $2$ and \[ex'(T)=\sum_{j=0}^d ex'_j\]

We define a \emph{right-layered} tree representation by the definition of left-layered trees above with rule $(3)$ replaced by the following:\\

\indent\indent $(3')$ If $u$ and $v$ are siblings on the same level and $\gamma(u)=\gamma(v)$, and $\deg(u)\leq \deg(v)$, then $u\prec v$.

The excess of right-layered trees is defined in the same way as for left-layered trees.

\begin{obs}\label{lobstermash}
Notice that for lobsters ($2$-distant trees) 
\begin{align}
ex'(T)=ex(T) 
\end{align}
\end{obs}

Similarly for closer relations, let $s_j$ be the number of consecutive vertices on level $j$ with distance equal to $2$. That is, $s_j=n_j-ex_j-1$. Define the $surplus$ of $T$, denoted by $s(T)$, as
\[s(T)=\sum_{j=3}^d s_j\]

Thus, when $T$ is the tree in Figure \ref{root}$(a)$, $s(T)=9$ and when $T$ is the tree in Figure \ref{root}$(b)$, $s(T)=9$.

\subsection{Range-Relaxed Labelings}

\begin{thm}\label{layered}
Every tree $T$ of size $m$ has a range-relaxed graceful labeling with vertex labels in the range $0, \dots, m+ex(T)$.
\end{thm}

To summarize the labeling, we assign the bottom level, $L_d$, of a left-layered tree $T$ by consecutive labels, starting at zero, from right to left.  Next, for the level of distance two from the bottom level, we assign labels consecutively from right to left, starting at the label which is one larger than the last label of the bottom level plus the excess of that level, that is, $n_d+ex_d+1$. We continue this way, going up, for every other level. To label the remaining levels, we start by labeling the topmost level which had not been labeled, and continue with the labeling scheme moving downward by two levels at a time, this time assigning labels consecutively from left to right.

\begin{proof}
We assign labels on the vertices of left-layered tree $T$ with size $m$ and diameter $d$.

The vertices on level $L_d$ receive labels from the interval $[0,n_d-1]$.

For $1\leq i \leq \left\lfloor\frac{d}{2}\right\rfloor$, the vertices on level $L_{d-2i}$ receive the labels from the interval
\[\left[\sum_{j=1}^i(n_{d+2-2j}+ex_{d+2-2j}), \sum_{j=1}^i(n_{d+2-2j}+ex_{d+2-2j})+n_{d-2i}-1\right]\]

We will use the parameter \[B=\sum_{j=0}^{\left\lfloor\frac{d}{2}\right\rfloor}(n_{d-2j}+ex_{d-2j})\] to define the labels on the remaining levels.

For $1\leq i \leq \left\lceil\frac{d}{2}\right\rceil$, the vertices on level $L_{d+1-2i}$ receive labels from the interval
\[\left[B+\sum_{j=i}^{\left\lceil\frac{d}{2}\right\rceil}(n_{d+1-2j}+ex_{d+1-2j})-n_{d+1-2j}, B+\sum_{j=i}^{\left\lceil\frac{d}{2}\right\rceil}(n_{d+1-2j}+ex_{d+1-2j})-1  \right]\]

On levels $L_d,L_{d-2}, L_{d-4}, \dots$, assign labels from left to right in descending order, and on levels $L_{d-1}, L_{d-3}, L_{d-5}, \dots$, assign labels from left to right in ascending order. Recall that $T$ is a left-layered tree and in this representation, the fact that $v_i^j \prec v_{i+1}^j$ implies that they have the same parent or that the parent of $v_i^j$ is located in $L_{j-1}$, to the left of the parent of $v_{i+1}^j$. Hence, the edges connecting vertices of $L_{j-1}$ and $L_j$ do not cross. In addition, the labels of the vertices in $L_{j-1}$ are in ascending (or descending) order, while the vertices of $L_j$ are in descending (or ascending) order. Therefore, there are no repetitions of weights.

Notice that the weights of the edges connecting vertices of $L_{j-1}$ and $L_j$ are in strictly ascending order. To see this, suppose that $t\in\{1,2,\dots d-1\}$; let $x \in L_{t-1}, y,z\in L_t$, and $w \in L_{t+1}$, such that $xy$ has the largest weight on the edges connecting vertices of $L_{t-1}$ and $L_t$, and $zw$ has the smallest weight among edges connecting vertices of $L_t$ and $L_{t+1}$. When $x<y$, $x=w+ex_{t+1}+1$ and $z=y-ex_{t+1}$, implying $z-w=y-ex_{t+1}-(x-ex_{t+1}-1)=y-x+1$. When $x>y$, $w=x+ex_{t+1}+1$ and $z=y+ex_{t+1}$, implying $w-z=x+ex_{t+1}+1-(y+ex_{t+1})=x-y+1$. Therefore, the edges $xy$ and $zw$ have consecutive weights, which means that the weights induced by this assignment are all distinct.

The largest label is
\begin{align*}
&B-1+\sum_{j=1}^{\left\lceil\frac{d}{2}\right\rceil}(n_{d+1-2j}+ex_{d+1-2j})\\
&=\sum_{j=0}^{\left\lfloor\frac{d}{2}\right\rfloor}n_{d-2j}+\sum_{j=1}^{\left\lceil\frac{d}{2}\right\rceil}n_{d+1-2j}-1+\sum_{j=0}^{\left\lfloor\frac{d}{2}\right\rfloor}ex_{d-2j}+\sum_{j=1}^{\left\lceil\frac{d}{2}\right\rceil}ex_{d+1-2j}\\
&=\sum_{j=0}^d n_j-1 + \sum_{j=0}^d ex_j\\
&=m + ex(T).
\end{align*}

Therefore, the assignment is a range-relaxed graceful labeling of $T$.
\end{proof}

In Figure \ref{example} we show a range-relaxed graceful labeling of the tree in Figure \ref{root}$(a)$. The excess of each level is shown by the black vertices. Note that $B=21$.

\begin{figure}[ht]
\begin{center}
\begin{tikzpicture}[scale=.75]
\tikzstyle{vert}=[circle,fill=black,inner sep=3pt]
\tikzstyle{overt}=[circle, draw, inner sep=3pt]



\node[overt, label=right:\tiny{$0$}] (u1) at (7,1) {};
\node[vert, label=right:\tiny{$42$}] (u2) at (7,2) {};
\node[vert, label=right:\tiny{$5$}] (u3) at (7,3) {};
\node[vert, label=right:\tiny{$35$}] (u4) at (7,4) {};
\node[vert, label=right:\tiny{$10$}] (u5) at (7,5) {};
\node[vert, label=right:\tiny{$28$}] (u6) at (7,6) {};
\node[vert, label=right:\tiny{$16$}] (u7) at (7,7) {};
\node[overt, label=right:\tiny{$23$}] (u8) at (7,8) {};
\node[overt, label=right:\tiny{$18$}] (u9) at (7,9) {};
\node[overt, label=right:\tiny{$22$}] (u10) at (7,10) {};
\node[overt, label=right:\tiny{$20$}] (u11) at (7,11) {};
\node[overt, label=right:\tiny{$21$}] (u12) at (7,12) {};

\node[overt, label=left:\tiny{$19$}] (u13) at (6,9) {};

\node[overt, label=left:\tiny{$17$}] (u14) at (2,7) {};

\node[overt, label=left:\tiny{$25$}] (u15) at (1,6) {};
\node[overt, label=left:\tiny{$26$}] (u16) at (2,6) {};
\node[vert, label=above:\tiny{$27$}] (u17) at (3,6) {};

\node[overt, label=left:\tiny{$13$}] (u18) at (1,5) {};
\node[vert, label=left:\tiny{$12$}] (u19) at (2,5) {};
\node[vert, label=right:\tiny{$11$}] (u20) at (3,5) {};

\node[overt, label=left:\tiny{$32$}] (u21) at (1,4) {};
\node[overt, label=left:\tiny{$33$}] (u22) at (2,4) {};
\node[overt, label=right:\tiny{$34$}] (u23) at (3,4) {};

\node[overt, label=left:\tiny{$8$}] (u24) at (2,3) {};
\node[overt, label=right:\tiny{$7$}] (u25) at (3,3) {};
\node[vert, label=left:\tiny{$6$}] (u26) at (6,3) {};

\node[overt, label=left:\tiny{$38$}] (u27) at (3,2) {};
\node[overt, label=left:\tiny{$39$}] (u28) at (4,2) {};
\node[overt, label=left:\tiny{$40$}] (u29) at (5,2) {};
\node[vert, label=left:\tiny{$41$}] (u30) at (6,2) {};

\node[overt, label=left:\tiny{$2$}] (u31) at (5,1) {};
\node[overt, label=left:\tiny{$1$}] (u32) at (6,1) {};

  \draw[color=black] 
(u1)--(u2)--(u3)--(u4)--(u5)--(u6)--(u7)--(u8)--(u9)--(u10)--(u11)--(u12)

(u13)--(u10) (u14)--(u8) (u15)--(u14) (u16)--(u14) (u17)--(u7) (u18)--(u16) (u19)--(u16) (u20)--(u17) (u21)--(u18) (u22)--(u19) (u23)--(u20) (u24)--(u23) (u25)--(u23) (u26)--(u4) (u27)--(u25) (u28)--(u26) (u29)--(u26) (u30)--(u26) (u31)--(u29) (u32)--(u30)

;

\end{tikzpicture}
\caption{Range-relaxed graceful labeling of a left-layered tree}
\label{example}
\end{center}
\end{figure}
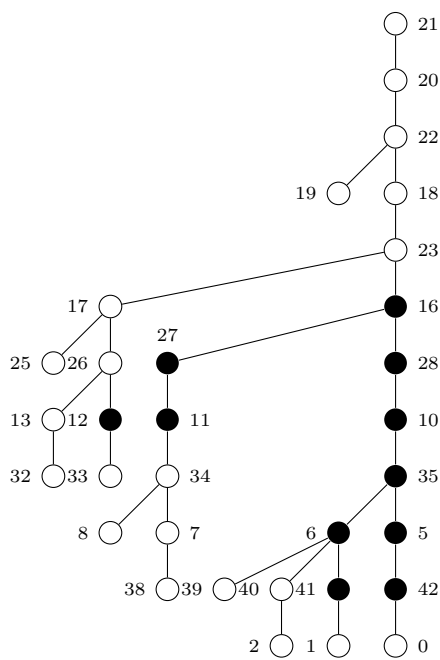

As a consequence of this theorem we have the following corollary, first proven by Rosa in 1966 \cite{Rosa}

\begin{cor}\cite{Rosa}
Caterpillars are graceful.
\end{cor}

\begin{proof}
Notice that for any caterpillar $C$, $ex(C)=0$.
\end{proof}

\begin{thm}
If $T$ is a lobster of size $m$ and diameter $d$, then the maximum vertex label, $v_{max} \leq \frac{3}{2}m-\frac{1}{2}d$.
\end{thm}

\begin{proof}
Notice that for any $2$-distant tree $T$,
\begin{align}
s(T)+ex'(T)+d=m \label{levelsum}
\end{align}
and since any pair of consecutive vertices on a given level with distance greater than $2$ must be preceded by ancestors one level above, whose distance is $2$, 
\begin{align}
s(T) \geq ex'(T).\label{lobstex}
\end{align}

We apply \ref{levelsum} and \ref{lobstex} to the maximum vertex label given by Theorem \ref{layered}. 
\end{proof}

\subsection{Lobster Shells}

Next, we prepare a structure which, though interesting in its own right, will be used to apply BH labelings to edge-relaxed labelings.

\begin{defn}\label{shell}
A lobster $T$ is a \emph{shell} (sometimes called a \emph{lobster shell}) if there exists a longest path $P$, so that all vertices not on $P$ with a neighbor on $P$ have degree two.

\end{defn}

For any right-layered represented lobster $T$ with longest path $P$, the \emph{shell of }$T$ is found by contracting all but one leaf vertex adjacent to vertices of distance one to $P$, and contracting all leaf vertices adjacent to $P$.

In Figure \ref{exshell} we show an example of a right-layered lobster and its shell. Black vertices are contracted.

\begin{figure}[ht]
\begin{center}
\begin{tikzpicture}[scale=.75]
\tikzstyle{vert}=[circle,fill=black,inner sep=3pt]
\tikzstyle{overt}=[circle, draw, inner sep=3pt]
  
\node[overt, label=right:\tiny{}] (u1) at (6,1) {};
\node[overt, label=right:\tiny{}] (u2) at (6,2) {};
\node[overt, label=right:\tiny{}] (u3) at (6,3) {};
\node[overt, label=right:\tiny{}] (u4) at (6,4) {};
\node[overt, label=right:\tiny{}] (u5) at (6,5) {};
\node[overt, label=right:\tiny{}] (u6) at (6,6) {};

\node[vert, label=right:\tiny{}] (u7) at (4,4) {};
\node[vert, label=right:\tiny{}] (u8) at (5,4) {};

\node[overt, label=right:\tiny{}] (u9) at (1,3) {};
\node[overt, label=right:\tiny{}] (u10) at (4,3) {};

\node[overt, label=right:\tiny{}] (u11) at (1,2) {};
\node[vert, label=right:\tiny{}] (u12) at (2,2) {};
\node[vert, label=right:\tiny{}] (u13) at (3,2) {};
\node[overt, label=right:\tiny{}] (u14) at (4,2) {};
\node[overt, label=right:\tiny{}] (u15) at (5,2) {};

\node[vert, label=right:\tiny{}] (u16) at (4,1) {};
\node[overt, label=right:\tiny{}] (u17) at (5,1) {};

\draw[color=black] 
(u1)--(u2)--(u3)--(u4)--(u5)--(u6) (u7)--(u5) (u8)--(u5) (u9)--(u4) (u10)--(u4) (u11)--(u9) (u12)--(u10) (u13)--(u10) (u14)--(u10) (u15)--(u3) (u16)--(u15) (u17)--(u15)

;

\node[overt, label=right:\tiny{}] (u1) at (12,1) {};
\node[overt, label=right:\tiny{}] (u2) at (12,2) {};
\node[overt, label=right:\tiny{}] (u3) at (12,3) {};
\node[overt, label=right:\tiny{}] (u4) at (12,4) {};
\node[overt, label=right:\tiny{}] (u5) at (12,5) {};
\node[overt, label=right:\tiny{}] (u6) at (12,6) {};

\node[overt, label=right:\tiny{}] (u7) at (9,3) {};
\node[overt, label=right:\tiny{}] (u8) at (10,3) {};

\node[overt, label=right:\tiny{}] (u9) at (9,2) {};
\node[overt, label=right:\tiny{}] (u10) at (10,2) {};
\node[overt, label=right:\tiny{}] (u11) at (11,2) {};

\node[overt, label=right:\tiny{}] (u12) at (11,1) {};

\draw[color=black] 
(u1)--(u2)--(u3)--(u4)--(u5)--(u6) (u7)--(u4) (u8)--(u4) (u9)--(u7) (u10)--(u8) (u11)--(u3) (u12)--(u11)

;

\end{tikzpicture}
\caption{A right-layered lobster and its shell}
\label{exshell}
\end{center}
\end{figure}
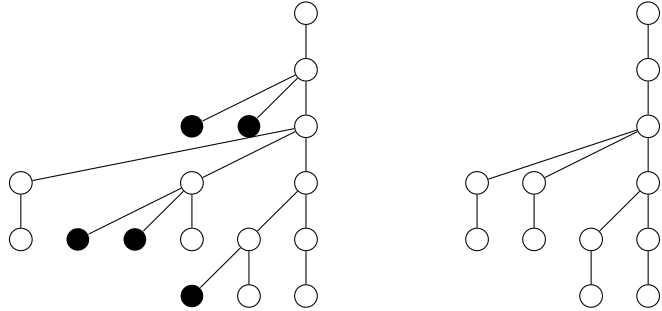

\begin{prop}\label{perfmatch}
For any lobster $T$ of order $n$, the $shell$ of $T$ has a perfect matching if $n$ is even or a $n-1$-matching (a matching covering $n-1$ vertices) if $n$ is odd. 
\end{prop}
\begin{proof}
If the shell of $T$ is a path, then the statement is true. Otherwise, the shell of $T$ is a maximal path incident to some number of branches composed of two edges. The statement is easy to show by induction on the number of such branches.
\end{proof}

We now review the labeling shown in \cite{BH}. For any lobster $T$ of even order $n$ with a perfect matching $M$, a BH labeling can be found by contracting the edges of $M$ to form the \emph{contree} $T'$, which is a caterpillar, or $1$-distant tree. For any graceful labeling $f$ of $T'$, label the vertices of $T'$ by $2f$, that is, multiply each vertex label by $2$, and then expand the edges back to form $T$ where only half the vertices are labeled such that the edges of $T'$ correspond to the edges of $T$ not in $M$. To label the remainder of the vertices of $T$, suppose $u$ is a labeled vertex and notice that $u$ is adjacent to an unlabeled vertex $v$ by an edge from $M$. Label $v$ by $n-1-2f(u)$ in such a way that for all the edges of $T$ which are not in $M$, both endpoints are to be labeled by even numbers or odd numbers so that these edges have even weight.

\begin{prop}\label{BHRelaxed}
For any lobster $T$ of order $n$ and non-negative integer $l$, if there exists a right-layered representation of $T$ so that the shell of $T$ can be found by $l$ contractions, then 
\[
gs(T)\geq \left\{\begin{array}{ll}
n-l-1, & \text{ if } n-l \text{ is even}\\
n-l-2, & \text{ if } n-l \text{ is odd}
\end{array}\right.
\]

\end{prop}
\begin{proof}
Perform $l$ contractions to produce the $shell$ of $T$, $S(T)$, and notice by Proposition \ref{perfmatch} that if $n-l$ is even, then $S(T)$ contains a perfect matching. By Theorem \ref{Morgan}, $S(T)$ accepts a graceful labeling. Expanding the previously contracted edges and labeling the new vertices uniquely by labels from the set $\{n-l+1, \dots, n\}$, we create at most $l$ repetitions of edge weights. If $n-l$ is odd, then $S(T)$ contains an $n-1$-matching. Call the graph obtained by performing one additional contraction on a leaf of $S(T)$ , $S'(T)$. By Theorem \ref{Morgan}, $S'(T)$ accepts a graceful labeling. Expanding the previously contracted edges and labeling the new vertices uniquely by labels from the set $\{n-l-2, \dots, n\}$, we create at most $l+1$ repetitions of edge weights.
\end{proof}

We say that a tree $T$ is \emph{pretty graceful} if there exists a vertex $v$ such that the contraction $T\backslash v$ is graceful, which leads us to the following amusing conclusion.

\begin{cor}
Lobster shells are pretty graceful.
\end{cor}

However, by more careful analysis, we can do a little better. We shall use the following result of Burzio and Ferrarese \cite{BF}, as stated by Superdock \cite{Superdock}.

\begin{defn}\cite{SZ}
Let $S$ and $T$ be trees and let $u,v$ be vertices of $S$ and $T$, respectively. Replace each vertex of $S$, other than the exceptional vertex $u$, by a copy of $T$ by identifying each vertex of $S$ with the vertex corresponding to $v$ in the distinct copy of $T$. Denote the resulting tree by $S\Delta_{+1}T$.
\end{defn}

\begin{thm}\cite{SZ}
If $S$ of order $n_S$ and $T$ are trees with graceful labelings $f$ and $g$, where $f(u)=n_S-1$ and $g(v)=0$, then $S\Delta_{+1}T$ is graceful.
\end{thm}

Suppose $S$ and $T$ have orders $n_S$ and $n_T$, respectively, with graceful labelings $f$ and $g$. The construction used to prove the above theorem requires $n_S-1$ copies of $T$, each substituted for a vertex of $S$ other than $u$. Let $(A,B)$ be the natural bipartition of $T$ with $v\in A$. The labeling function follows, other than the label on $u$, which is $(n_S-1)n_T$.
  \[ g_i(x) = \left\{\begin{array}{ll}
    in_T+g(x), & \text{ if } x\in A\\
    (n_S - i -2)n_T+g(x), & \text{ if } x \in B\\
      \end{array}\right. \]

Burzio and Ferrarese  \cite{BF} improved this method and called it the \emph{generalized $\Delta_{+1}$} construction by noticing that for any two adjacent vertices of $S$ into which we substitute copies of $T$, we may connect two such copies of $T$ by an edge between {\bf any} two vertices that correspond to the same vertex in each copy of $T$. The exceptional vertex $u$ must still be adjacent to the vertices corresponding to the fixed vertex $v$ of $T$.


\begin{thm}\label{gracefulshells}
Lobster shells are graceful.
\end{thm}

\begin{proof}
By Proposition \ref{perfmatch}, either a shell has a perfect matching or an almost perfect matching. Suppose the latter. Let $P$ be a path of diameter length. By induction on the number of vertices, it is easy to show that there exists an almost perfect matching $M$ which does not cover an end vertex $u$ of $P$. However, this means we can apply the generalized $\Delta_{+1}$ construction of \cite{BF} with $u$ as the exceptional vertex, the caterpillar formed by contracting the edges of $M$ as $S$ (labeled gracefuly with largest label on $u$), and $P_2$ as $T$.
\end{proof}

An alternate simple argument, which avoids the $\Delta_{+1}$ construction, provided by an anonymous referee, is as follows:

If a lobster shell has a perfect matching, then the result is known, thus assume otherwise. By definition of a lobster shell, an endpoint of $P$ must be a pendant vertex on the shell. Say $u$ is such a vertex. If we delete $u$, then we get a shell with a perfect matching, say $M$. After contracting edges of $M$, we get a caterpillar where the neighbor $v$ (say) of $u$ is an endpoint of the spine. It is then possible to find a graceful labeling of the caterpillar where $v$ is labeled by $0$. Then the label of $u$ becomes automatic in the shell. 

\medskip

Rosa and \v{S}ir\'{a}\v{n} \cite{RS} called a tree a $m$\emph{-comet} if its vertex set $V$ admits a bipartition $(A,B)$ with $A=\{u, w_1,w_2, \dots, w_m\}$ and $B=\{v_1,\dots, v_m\}$ such that $E=\{u_iv_i,v_iw_i:1\leq i \leq m\}$. The vertex $u$ will be called the \emph{central} vertex of the $m$-comet $T$. For a given value of $i\in\{1,2, \dots, m\}$, the $u-w_i$ path is called a \emph{ray} of $T$. A tree $T=T_{s,t}'$ is called a \emph{broken comet} or \emph{stardust}, if it can be obtained from an $s$-comet by attaching $t$ pendant edges to the central vertex of the comet. For any pair of disjoint trees $T_1$ and $T_2$ with distinguished vertices $u_1$ and $u_2$, we write $T_1 \circ T_2$ to denote the tree obtained by identifying $u_1$ with $u_2$, and call it a \emph{vertex amalgamation} of $T_1$ and $T_2$.

Recall that a \emph{bipartite labeling} of a tree $T$ of size $m$ is a bijection $f:V\rightarrow \{0,1,\dots, m\}$ such that there is an integer $\lambda$ such that if $f(u)\leq \lambda < f(v)$, then $u$ and $v$ belong to different sets of the bipartition $(A,B)$ of $V$. For any bipartite labeling of $T$, the \emph{complementary labeling} $\bar{f}$ is defined by $\bar{f}(v)=m-f(v)$ for each vertex $v\in T$. Let $A$ be the part of the bipartition $(A,B)$ of $T$ for which $f^{-1}(0)\in A$. The \emph{reverse labeling} $\hat{f}$ is given by
      \[ \hat{f}(v) = \left\{\begin{array}{ll}
    \left|A\right|-1-f(v), & \text{ if } v\in A\\
    2\left|A\right| + \left|B\right|-1-f(v), & \text{ if } v \in B
    \end{array}\right. \]

We will use the following Lemma and Proposition from \cite{RS}

\begin{lem}\cite{RS}\label{amalgam}
Let $T_1$ and $T_2$ be vertex-disjoint trees with distinguished vertices $u_1\in V(T_1)$ and $u_2\in V(T_2)$. Assume that there are bipartite labelings $f_1$ and $f_2$ of $T_1$ and $T_2$, respectively, such that $f_1(u_1)=f_2(u_2)=0$. Then there exists a bipartite labeling of $T_1 \circ T_2$ such that $\varepsilon(f)\geq \varepsilon(f_1) + \varepsilon(f_2)$. Consequently, for the $\alpha$-size of $T_1 \circ T_2$ we have $\alpha(T_1 \circ T_2)\geq \varepsilon(f_1)+\varepsilon(f_2)$.
\end{lem}

\begin{prop}\cite{RS}\label{prop4}
The $\alpha$-size of $T_{s,t}'$ is at least $\alpha'_{s,t}=\lfloor\frac{(5s+1)}{3}\rfloor +t$
\end{prop}

\begin{thm}\label{stardust}
Every lobster shell is an amalgamation of stardust. In particular, if $T$ is a lobster shell, then there exist stardust graphs $S_1, \dots, S_k$ with bipartite labelings $f_1,\dots, f_k$ so that $T=S_1 \circ S_2 \circ \dots, \circ S_k$  and there exists a bipartite labeling of $T$ so that $\alpha (T)\geq \varepsilon (f_1) + \dots + \varepsilon (f_k)$.
\end{thm}

\begin{proof}
We will apply the following labeling from the proof of Lemma \ref{amalgam} to compose stardust graphs. We call this labeling the \emph{Rosa-\v{S}ir\'{a}\v{n}}-labeling or just $RS$-labeling. For $i=1,2$ let $(A_i,B_i)$ be the bipartition of $V(T_i)$ for which $f_i^{-1}(0)=u_i\in A_i$. Define the labeling $f$ of $T$ as
    \[ f(v) = \left\{\begin{array}{ll}
    \hat{f}_1(v), & \text{ if } v\in A_1\\
    f_2(v) + \left|A_1\right|-1, & \text{ if } v \in A_2 \cup B_2\\
    \hat{f}_1(v)+\left|A_2\right|+\left|B_2\right|-1, & \text{ if } v \in B_1
    \end{array}\right. \]

If $T$ has diameter $4$, then $T$ is a comet and we use the $RS$-labeling on $T$. 

Suppose that $T$ has diameter $5$. Let $P$ be a path of maximum length and consider a right layered representation of $T$. Remove the edge on $P$ between $L_2$ and $L_3$, creating two comets. Let $T_1$ be the comet that contains $L_0$ and let $S_2$ be the comet that contains $L_5$. Let $f$ be the $RS$-labeling of $T_1$ and $g$ be the $RS$-labeling of $S_2$. Notice that under $\hat{f}$, the label on the central vertex of $T_1$ is $0$. Thus, we can amalgamate $P_2$ and $T_1$ by labeling one vertex, $x$, of $P_2$ by $0$ and the other, $y$, by $\left|T_1\right|$. Call the amalgamated graph $S_1$. Next, notice that $y$ is labeled $0$ under the complementary labeling of the reverse labeling of $f$, $\overline{\hat{f}}$. 

We label $S_2$ by $\hat{g}$ so that the central vertex receives the label $0$. Finally, notice that labeling $S_1$ by $\overline{\hat{f}}$ and $S_2$ by $\hat{g}$ produces the required labeling of $T=S_1\circ S_2$.

Under this labeling of $T$, the maximum label is found on $L_4$ and the label $0$ is on $L_3$. 

If $T$ has diameter $6$ (and for each unit increase in diameter) the above labeling can be iterated. Remove the edge on $P$ between $L_3$ and $L_4$, creating a shell of diameter $5$, $T_1$, and a comet $S_2$. Let $f$ be the labeling found for diameter $5$ shells and amalgamate $P_2$ to $T_1$ with labeling $f$ and call the resulting graph $S_1$. Let $g$ be the $RS$-labeling of $S_2$. We can now amalgamate $S_2$ to the resulting graph, with labeling $\overline{f}$ on $S_1$ and $\hat{g}$ on $S_2$. The rest of the proof proceeds similarly as an induction on diameter.

\end{proof}

The $\alpha$-size of stardust from \cite{RS} was calculated in Proposition $\ref{prop4}$ and allows us to state the following bound.

\begin{cor}
For any lobster shell $T$ of size $m$, $\alpha(T) \geq \left\lfloor\frac{5m+2}{6}\right\rfloor$.
\end{cor}
\begin{proof}
We use the simple property that for any sequence of non-negative real numbers, $x_1, \dots, x_j$,
\begin{align}
\left\lfloor x_1\right\rfloor + \left\lfloor x_2\right\rfloor + \dots + \left\lfloor x_j\right\rfloor \geq \left\lfloor\sum_{i=1}^j x_i \right\rfloor +j-1 \label{eq1}
\end{align}
We decompose $T$ into stardust graphs and a comet as in the Theorem \ref{stardust}, so that $T=S_1 \circ S_2 \circ \dots, \circ S_{k-1}\circ T_k$.
Using Proposition $4$ of \cite{RS}, we calculate the $\alpha$-size of each stardust graph and the comet as at least
\[\left(\sum_{i=1}^{k-1}\left\lfloor\frac{5m_i+2}{6} \right\rfloor \right) + \left\lfloor \frac{5m_k+2}{6} \right\rfloor\]
where $m_i$ is the size of $S_i$. 

Note that $k=d-3$, where $d$ is the diameter of $T$. Applying \ref{eq1} completes the proof.
\end{proof}

In light of the result from \cite{RS} that for any comet $T$, $\alpha(T)\leq \frac{5(m+9)}{6}$, the above corollary shows that lobster shells have maximum $\alpha$-size with respect to the multiplicative constant.

We can extend Definition \ref{shell} for any three-distant tree $T$ with longest path $P$ by defining the \emph{shell of a three distant tree} as the graph found by contracting all branches not on $P$ with size less than $3$ and contracting all but one leaf vertex adjacent to vertices of distance two to $P$. 

In \cite{K}, the author showed

\begin{thm}\cite{K}
Every lobster with an almost perfect matching is graceful.
\end{thm}

We challenge the interested reader to prove the following natural extensions of results on lobsters to three-distant trees:

\begin{conj}
Every shell of a three-distant tree is graceful.
\end{conj}

\begin{conj}
Every three-distant tree with a perfect matching is graceful.
\end{conj}

\begin{conj}
Every three-distant tree with an almost perfect matching is graceful.
\end{conj}

\subsection{Edge-Relaxed Labelings}

We introduce some useful notation.

Let $T$ be a lobster and $P$ a longest path in $T$. Suppose that $T$ is represented as a right-layered rooted tree, where the root is an end-vertex of $P$. We say that an edge is \emph{on level $i$} if it is incident to the vertex of $P$ on $L_{i-1}$. An edge is of \emph{distance $0$} to $P$ if it is an edge of $P$. For $i\geq 1$, an edge is of \emph{distance $i$} to $P$ if it is incident to a vertex which is of distance $i-1$ from $P$.

\begin{itemize}
\item  $\theta_i$ is the number of elements in the set of those weights of edges on level $i$ that occurred on levels above $i$ (that is, levels $j<i$)
\item $\theta=\sum_{i=0}^d \theta_i$
\item $d_i$ is the average degree over vertices on level $i$ which are of distance $1$ from $P$ and have at least one neighbor not on $P$
\item $s(i,i+1)$ is the number of pairs of consecutive vertices at distance $2$ from each other (surplus) on level $i$ which are also vertices of distance $1$ or $0$ to $P$, and the number of pairs of consecutive vertices at distance $2$ (surplus) on level $i+1$ which are also vertices of distance $2$ to $P$
\item $\alpha(i,i+1)$ is the number of distinct weights on levels $i$ and $i+1$ so that edges on level $i$ are distance $0$ or $1$ to $P$ and edges on level $i+1$ are of distance $2$ to $P$
\item $m(i,i+1)$ is the number of edges on levels $i$ and $i+1$ so that edges on level $i$ are of distance $0$ or $1$ to $P$ and edges on level $i+1$ are of distance $2$ to $P$
\item $p_i$ is the number of edges on level $i$ which are leaves of distance $1$ to $P$.
\end{itemize}

Since $s(i,i+1)$ counts pairs of consecutive incident edges and $ex_{i+1}$ counts pairs of consecutive non-incident edges, the following relation holds for any right-layered tree $T$
\begin{align}
s(i,i+1) + ex_{i+1} +1 = m(i,i+1)\label{relation}
\end{align}

\begin{thm}\label{edge-relaxedbip}
 For any lobster $T$ with $m$ edges and diameter $d$, $\alpha(T)\geq \max\{\frac{3m-d+6}{4},\frac{5m+d+15}{8}\}$.
\end{thm}
\begin{proof}
We label the right-layered tree $T$ as follows. For odd $d$ and $l=\sum_{i \text{ odd}}\left|L_i\right|$ we label vertices consecutively from right to left starting from level $d$ and decreasing levels by two until level $1$, from the interval
$\left[ 0, l-1\right].$
Then, from level $0$ to level $d-1$, label vertices consecutively from left to right from the interval
$\left[ l, \left(\sum_{i=0}^d\left|L_i\right|\right)-1\right].$

For even $d$ and $l=\sum_{i \text{ even}}\left|L_i\right|$ we label the vertices consecutively from right to left starting from level $d$ and decreasing levels by one until level $0$, from the interval
$\left[ 0,  l-1\right].$
Then, from level $1$ to level $d-1$, label vertices consecutively from left to right from the interval
$\left[ l, \left(\sum_{i=0}^d\left|L_i\right|\right)-1\right].$

\begin{claim}\label{norepeat}
Every edge weight of $T$ may be repeated at most once and only on consecutive levels.
\end{claim}
\begin{proof}
Let $P$ be a maximum path of the right-layered representation of $T$ beginning at the root with vertices $x_0,x_1, \dots, x_d$. Call the above labeling function $f$, and for any level $L_i, 0\leq i \leq d$, let $f(L_i)$ denote the labels of vertices on $L_i$. We say an edge $e$ is \emph{on level $i$} if $e$ joins vertices on levels $i-1$ and $i$.

Notice that by definition of $f$, $f(L_i)$ has no repetitions for any $i$.

Let $l_i=\left|L_i\right|$ and choose $i,j$ so that $1\leq i+1<j\leq d$. We show that no weight of an edge from $L_i$ can be repeated on $L_j$. Assume $i$ and $j$ are of the same parity and that $f(x_i)$ is the minimum label on level $i$. Notice that
\[f(L_i)= \left[ f(x_i), f(x_i)+l_i-1\right],\] \[f(L_j)= \left[ f(x_j), f(x_j)+l_j-1\right]\]
and $f(x_j)>f(x_i)+l_i-1$. Similarly, if $f(x_i)$ is the maximum label on level $i$, we have \[f(L_i)= \left[ f(x_i)-l_i+1, f(x_i)\right],\] \[f(L_j)= \left[ f(x_j)-l_j+1, f(x_j)\right]\] and $f(x_j)<f(x_i)-l_i+1$.

In either case, by considering $f(L_{i-1})$ and $f(L_{j-1})$, we can show that the weights of edges on $L_i$ and $L_j$ are distinct. That is, if $f(x_i)$ is the maximum label on level $i$, then we have $f(x_j)<f(x_i)-l_i+1$ and $f(x_{j-1})>f(x_{i-1}) + l_{i-1}-1$. Then the difference between the maximum edge-label for edges with vertices on levels $i-1$ and $i$ and the minimum label for edges with vertices on levels $j-1$ and $j$ is given by 

\begin{align*}
&(f(x_{j-1})-f(x_j))-((f(x_i)-l_{i-1}+1)-(f(x_i)+l_i-1))\\
&=(f(x_{j-1}-f(x_{i-1})+(f(x_i)-f(x_j))-l_i-l_{i-1}+2\\
&>(l_{i-1}-1)+(l_i-1)-l_i-l_{i-1}+2=0.
\end{align*}

If $i$ and $j$ are of opposite parity, assume without loss of generality that $f(x_i)$ is the minimum label on level $i$. Notice that
\[f(L_i)= \left[ f(x_i),  f(x_i)+l_i-1\right],\] \[f(L_{j-1})= \left[ f(x_{j-1},  f(x_{j-1})+l_{j-1}-1\right]\]
and $f(x_{j-1})>f(x_i)+l_i-1$. Again, considering $f(L_{i-1})$ and $f(L_{j})$, we see that the weights on $L_i$ and $L_j$ are distinct.

Suppose next that $0\leq i \leq d-1$ and consider the weights on $L_i$ and $L_{i+1}$. Since no edge weight can be repeated on a given level, any edge weight on level $L_i$ can be repeated at most once on level $L_{i+1}$, which proves the claim.
\end{proof}

\begin{claim}\label{surplusdegree-excess}
$s(i,i+1)\geq (d_i-1) ex_{i+1}+p_i$
\end{claim}

\begin{proof}
 If a pair of vertices $u,v$ on level $i$ with distance $0$ or $1$ to $P$ contribute to the surplus on level $i$, and each has a neighbor on $L_{i+1}$, then $u$ and $v$ have neighbors which contribute to the excess of $L_{i+1}$. Furthermore, the neighbors of $u$ on $L_{i+1}$ contribute to the surplus on $L_{i+1}$.
\end{proof}

Combining \ref{relation} with the above claim produces

\begin{align}
p_i+d_i \times ex_{i+1}\leq m(i,i+1)-1 \label{ineq1}
\end{align}

\begin{claim}\label{theta}
$\theta_i \leq \left\lceil\frac{d_i-1}{d_i}ex_{i+1}\right\rceil$
\end{claim}
\begin{proof}
Notice that for every pair of consecutive vertices $u,v$ on $L_i$, each of distance one to $P$, if $u$ and $v$ have descendants, then the pair $u,v$ corresponds to some pair of descendants on $L_{i+1}$ which contribute $1$ to $ex_{i+1}$. Furthermore, in a right-layered tree, the vertices of $L_i$ with distance one to $P$ are unique as incident vertices to edges of $L_i$ which may contribute to $\theta_i$ in repeating weights of $L_{i-1}$.

Notice that consecutive vertices of $L_{i+1}$ of distance $2$ away from each other, and distance $2$ from $P$, may be incident to edges of $L_i$ with weights that occured on $L_{i-1}$. Moreover, these weights are consecutive. However, for every consecutive pair of vertices of $L_{i+1}$ of distance more than $2$ away from each other, and distance $2$ from $P$, the weights of the corresponding edges have a difference of $2$.

Also, note that the minimum weight of edges of $L_{i-1}$ that are incident to edges of $L_i$ cannot be repeated since such a weight is on an edge $e$, which may only be incident to the same vertex $v$ as the edge $f$ with the minimum weight of $L_i$, and the other vertices incident to $e$ and $f$ must have different labels.

Suppose next that $d_i$ is an integer and for every vertex $v$ of $L_i$ of distance $1$ from $P$ with at least one neighbor not on $P$, $\deg(v)=d_i$. Call this the uniform case, and notice in light of the above observations, $\theta_i=\left\lceil\frac{d_i-1}{d_i}ex_{i+1}\right\rceil$. Moreover, in a right-layered tree, the degree sequence of vertices of $L_i$ but not of $P$, with neighbors on $L_{i+1}$, are monotonically increasing so that the number of weights that are skipped on level $i$ is the same as in the uniform case, though the skips in weights may occur between edges farther to the left. This observation completes the proof. 
\end{proof}

Notice that by Claim \ref{theta} and formula \eqref{ineq1} we can write
\begin{align}
\alpha(i,i+1)\geq m(i,i+1)-\theta_{i+1}\geq m(i,i+1)-\left\lceil\frac{d_i-1}{d_i}ex_{i+1}\right\rceil\nonumber\\
\geq m(i,i+1)-\left\lceil\frac{(d_i-1)(m(i,i+1)-1-p_i)}{d_i^2}\right\rceil\label{alphalevelgen}
\end{align}

The above term is minimized when $d_i=2$ and we obtain the bound
\begin{align} 
\alpha(i,i+1)\geq  m(i,i+1)-\left\lceil\frac{m(i,i+1)-1-p_i}{4}\right\rceil\label{alphalevel}
\end{align}

If $m(i,i+1)$ is even under the assumption that $d_i=2$, then notice that $p_i\geq 1$, and inequality \eqref{alphalevel} can be rewritten as

\[
\alpha(i,i+1)\geq \left\{\begin{array}{ll}
\frac{3}{4}m(i,i+1)+\frac{1}{2}, \text{ when } m(i,i+1)\equiv 0 \pmod{4}\\[0.5ex]
 \frac{3}{4}m(i,i+1), \text{ when } m(i,i+1)\equiv 2 \pmod{4}
\end{array}\right.
\]

If $m(i,i+1)$ is odd, inequality \eqref{alphalevel} can be rewritten as
\[
\alpha(i,i+1)\geq \left\{\begin{array}{ll}
\frac{3}{4}m(i,i+1)+\frac{1}{4}, \text{ when } m(i,i+1)\equiv 1 \pmod{4}\\[0.5ex]
 \frac{3}{4}m(i,i+1)-\frac{1}{4}, \text{ when } m(i,i+1)\equiv 3 \pmod{4}
\end{array}\right.
\]

Furthermore, edges of distance $0$ or $1$ from $P$ on levels $0,1,$ and $d-1$, have weights that are never repeated.

Thus, the ``worst case'' for the number of distinct edge weights of any lobster $T$ in our labeling is one where $m(i,i+1)$ is congruent to $3 \pmod 4$, producing the number of distinct weights of $T$ as
\[\alpha(T)\geq 1+ \sum_{i=0}^{d-1}{\alpha(i,i+1)}\geq \frac{3}{4}(m-3)-\frac{1}{4}(d-3)+3=\frac{3m-d+6}{4}\]

Note that since there are at least three edges with weights that are never repeated, we remove those weights from the the above sum and add them back. Notice that this bound is an improvement on \cite{RS} for small diameter trees, in particular, when $d<\frac{m+22}{7}$. With a few modifications, we can also improve the $\alpha$-size of lobsters with large diameter.

We observe that when $d_i=2$ and $p_i=0$ for all $i$, the number of levels with incident edges to $P$ is at most $\frac{m-3-d}{2}$. Thus we calculate
\[\alpha(T)\geq 1+ \sum_{i=0}^{d-1}{\alpha(i,i+1)}\geq \frac{3}{4}(m-3)-\frac{1}{4}\left(\frac{m-3-d}{2}-3\right)+3=\frac{5m+d+15}{8}\]

This bound is an improvement on \cite{RS} when $d>\frac{5m-65}{7}$.
\end{proof}

It is not difficult to find a lobster $T$ with a perfect matching such that any BH labeling of $T$ is not bipartite, as in the Figure \ref{notbip}.


\begin{figure}[ht]
\begin{center}
\begin{tikzpicture}[scale=1]
\tikzstyle{vert}=[circle,fill=black,inner sep=3pt]
\tikzstyle{overt}=[circle,fill=black!30, inner sep=3pt]

  \node[vert, label=left:\tiny{}] (u-1) at (1,1) {};
  \node[overt, label=right:\tiny{}] (u-2) at (2,1) {};
  \node[vert, label=right:\tiny{}] (u-3) at (3,1) {};
  \node[overt, label=left:\tiny{}] (u-4) at (4,1) {};
  \node[vert, label=right:\tiny{}] (u-5) at (5,1) {};
  \node[overt, label=right:\tiny{}] (u-6) at (6,1) {};
  \node[vert, label=right:\tiny{}] (u-7) at (7,1) {};
  \node[overt, label=right:\tiny{}] (u-8) at (8,1) {};

  \node[vert, label=left:\tiny{}] (v-1) at (1,1.5) {};
  \node[vert, label=right:\tiny{}] (v-2) at (1,2) {};
  \node[overt, label=below:\tiny{}] (v-3) at (2,1.5) {};
  \node[overt, label=left:\tiny{}] (v-4) at (2,2) {};
  \node[overt, label=right:\tiny{}] (v-5) at (2,2.5) {};
  \node[overt, label=above:\tiny{}] (v-6) at (2,3) {};
  \node[vert, label=above:\tiny{}] (v-7) at (3,2.5) {};
  \node[vert, label=above:\tiny{}] (v-8) at (3,3) {}; 
  \node[vert, label=above:\tiny{}] (v-9) at (3,3.5) {};
  \node[overt, label=above:\tiny{}] (v-10) at (4,3.5) {};
  \node[overt, label=above:\tiny{}] (v-11) at (4,4) {};    
  \node[vert, label=above:\tiny{}] (v-12) at (5,4) {};

   \draw[color=black] 
   (u-2)--(u-3) (u-4)--(u-5) (u-6)--(u-7)
   (v-3)--(u-3)
   (v-4)--(u-3)
   (v-7)--(u-4)
   (v-8)--(u-4)
   (v-10)--(u-5)
   (v-12)--(u-6)
;
  
  \draw[color=red, style=very thick]
  (u-1)--(u-2) (u-3)--(u-4) (u-5)--(u-6) (u-7)--(u-8)
  (v-1)--(v-3)
  (v-2)--(v-4) 
  (v-5)--(v-7)
  (v-6)--(v-8)
  (v-9)--(v-10)
  (v-11)--(v-12)
  ;
  
\end{tikzpicture}
\caption{A lobster with a perfect matching but no bipartite BH labeling}
\label{notbip}
\end{center}
\end{figure}
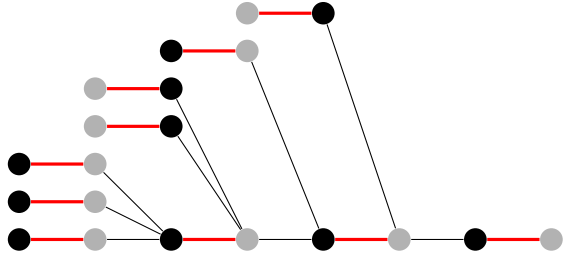

Although we cannot improve bounds on $\alpha(T)$ by applying BH labelings, the gracesize is another matter.

For any lobster $T$ of size $m$, let twice the size of a maximum matching on $T$ be $\nu(T)$, or just $\nu$ if $T$ is clear from context.

\begin{thm}\label{edge-relaxed}
 For any lobster $T$ with $m$ edges and diameter $d$, \[gs(T)\geq\max\left\{\frac{3}{4}m+\frac{d-\nu}{8}+\frac{3}{2},\nu\right\}.\]
\end{thm}

\begin{proof}
We continue from the proof of Theorem \ref{edge-relaxedbip} with the same terminology and notation. Observe that $T$ can be viewed as a lobster shell of order $\nu$ with $m-\nu$ amalgamated leaves. From this perspective, note that the number of levels of $T$ with incident paths of length $2$ which are not on $P$ is at most $\frac{m-d-(m-\nu)}{2}=\frac{\nu-d}{2}$. Summing \ref{alphalevel} over all such $i$, we obtain 
\[gs(T)\geq \frac{3}{4}(m-3)-\frac{1}{4}\left(\frac{\nu-d}{2}-3\right)+3\]
which is the first desired bound. The second bound is just Proposition \ref{BHRelaxed}.
\end{proof}

Note: The bound from the above theorem implies the following improvement

\begin{cor}
\begin{align}
&\text{If } \nu\geq \frac{3}{4}m, \text{ then } gs(T)\geq \frac{3}{4}m\\
&\text{If } \nu< \frac{3}{4}m, \text{ then } gs(T) \geq \frac{3}{4}m \text{ for } d\geq \nu-12
\end{align}
\end{cor}

\section{Remarks}
The improvement in the gracesize bound from Theorem \ref{edge-relaxed} comes at the cost of the labeling not necessarily being bipartite. However, this is the first instance of the use of a non-bipartite labeling in such a result, which we view as the correct approach since the conjectured bound from the GTC could not come from a bipartite labeling. A promising direction could be to find values of $d_i$ that produce the minimum gracesize of $T$ simultaneously by equation \ref{alphalevelgen} and Proposition \ref{BHRelaxed}. 

  Our approach shows improvements for range-relaxed graceful labeling and edge-relaxed graceful labelings of lobsters as a step towards Bermond's conjecture \cite{Bermond} that all lobsters are graceful. However, with more careful analysis of the excess of $k$-distant trees for $k>2$, analogous statements may be possible.

\section{Acknowledgement}

We want to express our gratitude to the anonymous referees for their thorough comments and suggestions. The high quality careful checking of our work together with ideas for changes greatly improved the content and style of this paper. Thanks!

\end{document}